\documentclass[10pt,letterpaper,dvips]{article}

 \usepackage{amssymb,latexsym,amsmath,amsthm}

 \usepackage{geometry}
\newtheorem{thm*}{Theorem}

\newtheorem{thm}{Theorem}

\newtheorem{dfn}{Definition}

\newtheorem{lemma}{Lemma}

\newtheorem{remark}{Remark}

\newtheorem{prop}{Proposition}

\begin{document}

\def\d{ \partial_{x_j} }
\def\Na{{\mathbb{N}}}

\def\Z{{\mathbb{Z}}}

\def\IR{{\mathbb{R}}}

\newcommand{\E}[0]{ \varepsilon}

\newcommand{\la}[0]{ \lambda}

\newcommand{\s}[0]{ \mathcal{S}}

\newcommand{\BO}[2]{ \left( #1 , #2 \right) }

\newcommand{\CO}[2]{ \left\langle #1 , #2 \right\rangle}

\newcommand{\R}[0]{ \IR\cup \{\infty \} }

\newcommand{\co}[1]{ #1^{\prime}}

\newcommand{\p}[0]{ p^{\prime}}

\newcommand{\m}[1]{   \mathcal{ #1 }}

\newcommand{ \W}[0]{ \mathcal{W}}

\newcommand{ \A}[1]{ \left\| #1 \right\|_H }

\newcommand{\B}[2]{ \left( #1 , #2 \right)_H }

\newcommand{\C}[2]{ \left\langle #1 , #2 \right\rangle_{  H^* , H } }

 \newcommand{\HON}[1]{ \| #1 \|_{ H^1} }

\newcommand{ \Om }{ \Omega}

\newcommand{ \pOm}{\partial \Omega}

\newcommand{\D}{ \mathcal{D} \left( \Omega \right)}

\newcommand{\DP}{ \mathcal{D}^{\prime} \left( \Omega \right)  }

\newcommand{\DPP}[2]{   \left\langle #1 , #2 \right\rangle_{  \mathcal{D}^{\prime}, \mathcal{D} }}

\newcommand{\PHH}[2]{    \left\langle #1 , #2 \right\rangle_{    \left(H^1 \right)^*  ,  H^1   }    }

\newcommand{\PHO}[2]{  \left\langle #1 , #2 \right\rangle_{  H^{-1}  , H_0^1  }}

 \newcommand{\HO}{ H^1 \left( \Omega \right)}

\newcommand{\HOO}{ H_0^1 \left( \Omega \right) }

\newcommand{\CC}{C_c^\infty\left(\Omega \right) }

\newcommand{\N}[1]{ \left\| #1\right\|_{ H_0^1  }  }

\newcommand{\IN}[2]{ \left(#1,#2\right)_{  H_0^1} }

\newcommand{\INI}[2]{ \left( #1 ,#2 \right)_ { H^1}}

\newcommand{\HH}{   H^1 \left( \Omega \right)^* }

\newcommand{\HL}{ H^{-1} \left( \Omega \right) }

\newcommand{\HS}[1]{ \| #1 \|_{H^*}}

\newcommand{\HSI}[2]{ \left( #1 , #2 \right)_{ H^*}}

\newcommand{\WO}{ W_0^{1,p}}
\newcommand{\w}[1]{ \| #1 \|_{W_0^{1,p}}}

\newcommand{\ww}{(W_0^{1,p})^*}

\newcommand{\Ov}{ \overline{\Omega}}

  \newcommand{\ve}[1]{  \mathbf{  #1 }}

  \author{Craig Cowan\footnote{Funded by NSERC.}}

\title{Uniqueness of solutions for elliptic systems and fourth order equations involving a parameter}
\maketitle

{\footnotesize
 \centerline{Department of Mathematics, Stanford University}
 \centerline{Building 380, Stanford, California}
 \centerline{650-723-2221}
 \centerline{ctcowan@stanford.edu}
} 

\begin{abstract}   We examine the equation 
\[ \Delta^2 u = \lambda f(u) \qquad \Omega, \] with either Navier or Dirichlet boundary conditions.  We show some uniqueness results under certain constraints on the parameter $ \lambda$.   We obtain similar results for the sytem

\begin{equation*}
 \left\{
\begin{array}{rrl}
-\Delta u &=& \lambda f(v) \qquad \Omega, \\
-\Delta v &=& \gamma g(u) \qquad \Omega, \\
u&=& v = 0 \qquad \pOm.
\end{array}
\right. 
\end{equation*}

\end{abstract}


  \section{Introduction}

  In this  note our main interest is in the uniqueness of solutions for some generalizations of the well studied second order problem $ -\Delta u = \lambda f(u)$.   We examine three generalizations:
  \begin{equation*}
\mbox{(Navier)} \qquad \qquad (N)_\lambda \qquad \left\{
\begin{array}{rll}
\Delta^2 u &=& \lambda f(u) \qquad \Omega \\
u &=& 0 \qquad  \qquad \pOm \\
\Delta u &=& 0 \qquad \qquad  \pOm,
\end{array}
\right.
\end{equation*}
\begin{equation*}
 \mbox{(Dirichlet)} \qquad  \qquad  (D)_\lambda \qquad \left\{
\begin{array}{rll}
\Delta^2 u &=& \lambda f(u) \qquad \Omega \\
u &=& 0 \qquad  \qquad \pOm \\
\partial_\nu u &=& 0 \qquad \qquad  \pOm,
\end{array}
\right.
\end{equation*} and

\begin{equation*}
\mbox{(System)} \qquad \qquad (P)_{\lambda,\gamma} \qquad \left\{
\begin{array}{rrl}
-\Delta u &=& \lambda f(v) \qquad \Omega \\
-\Delta v &=& \gamma g(u) \qquad \Omega \\
u &=& 0 \qquad \quad \pOm  \\
v&=& 0 \qquad \quad \pOm
\end{array}
\right. 
\end{equation*}

where $ \Omega$ is a bounded domain in $ \IR^N$ with smooth boundary,  $ \partial_\nu $ denotes the derivative on the boundary in the direction of the outward pointing normal $ \nu$ and where  $\gamma, \lambda >0$ are  parameters.   We assume that the nonlinearities $f$ and $g$ satisfies  either (R): $f>0$ on $ \IR$ with $ f$ smooth, increasing, convex, $ f(0)=1$  and $f$ is superlinear at $ \infty$  or $f$ satisfies (S): $ f>0$ on $ (-\infty,1)$ with $ f$ smooth, increasing, convex, $f(0)=1$  and $ f(1^-)=\infty$.   \\  
Some notations:  $ F(t):=\int_0^t f(\tau) d \tau, G(t):=\int_0^t g(\tau) d \tau$.  We say that $f$ is log convex provided $ t \mapsto log(f(t))$ is a convex function.  \\

\subsection{Preliminaries}

Given a  nonlinearity $ f$ which satisfies (R) or (S), the following equation 
 \begin{eqnarray*}
\hbox{$(Q)_{\lambda}$}\hskip 50pt \left\{ \begin{array}{lcl}
\hfill   -\Delta u  &=& \lambda f(u) \qquad \Omega \\
\hfill u&=& 0 \qquad\qquad \pOm 
\end{array}\right.
  \end{eqnarray*}
  is  now quite well understood whenever $ \Omega$ is a bounded smooth domain in $ \IR^N$. See, for instance, \cite{BV,Cabre,CC,EGG,GG,Martel,MP,Nedev,bcmr}. We now list the  properties one comes to expect when studying $(Q)_{\lambda}$.   It is well known that there exists a critical parameter  $ \lambda^* \in (0,\infty)$, called the extremal parameter,  such that for all $ 0<\lambda < \lambda^*$ there exists a smooth, minimal solution $u_\lambda$ of $(Q)_\lambda$.  Here minimal solution means in the pointwise sense.  In addition for each $ x \in \Omega$ the map $ \lambda \mapsto u_\lambda(x)$ is increasing in $ (0,\lambda^*)$.   This allows one to define the pointwise limit $ u^*(x):= \lim_{\lambda \nearrow \lambda^*} u_\lambda(x)$  which can be shown to be a weak solution, in a suitably defined sense, of $(Q)_{\lambda^*}$.  For this reason $ u^*$ is called the extremal solution. 
   It is also known that for $ \lambda >\lambda^*$ there are no weak solutions of $(Q)_\lambda$.  Also one can show the minimal solution $ u_\lambda$  is a semi-stable solution of $(Q)_\lambda$  in the sense that
   \[ \int_\Omega \lambda  f'(u_\lambda) \psi^2 \le \int_\Omega | \nabla \psi|^2, \qquad \forall \; \psi \in H_0^1(\Omega).\]
We now come to the results known for $(Q)_\lambda$ which we are interested in extending to $(N)_\lambda,(D)_\lambda$ and $(P)_{\lambda,\gamma}$. 
 \begin{itemize} \item In \cite{Martel} it was shown that if $f$ satisfies (R) then the  extremal solution $u^*$ is the unique weak solution of $(Q)_{\lambda^*}$.   This was extended to the case where $f$ satisfies (S), see \cite{moroc}.     
 \item  In \cite{jmq} and \cite{schmitt} a generalization of $(Q)_\lambda$ was examined.  They showed that if $f$ is suitably supercritical near $ u=\infty$ and if $ \Omega$ is a star shaped domain  then the minimal solution is the unique solution of $(Q)_\lambda$ for small $ \lambda$.    In \cite{A4} this was done for a particular nonlinearity $f$  which satisfies (S).   We remark that one can weaken the   star shaped assumption and still have uniqueness, see  \cite{shaaf},  but we do not pursue this approach here.  See  \cite{A1,A2,A3} for more results on this topic. 
 \end{itemize}  
 
We now turn our attention to the needed background and known results for $(N)_\lambda,(D)_\lambda$ and $(P)_{\lambda,\gamma}$. 
 
 \subsubsection*{Fourth order}
The problem $(Q)_\lambda$ is heavily dependent on the  maximum principle and hence this poses a major hurdle in the study of $(D)_\lambda$ since for general domains there is no maximum principle for $ \Delta^2$ with Dirichlet boundary conditions.   If one restricts their attention to the unit ball then one does have a weak maximum principle, see \cite{BOG}.  In this case there exists an extremal parameter $ \lambda^* \in (0,\infty)$ such that for all $ 0 < \lambda < \lambda^*$ there exists a smooth, minimal, stable solution $u_\lambda$ of $(D)_\lambda$.   By a stable solution we mean that 

\begin{equation} \label{szz}
 \int_\Omega \lambda f'(u_\lambda) \psi^2 \le \int_\Omega (\Delta \psi)^2, \quad \forall \;  \psi \in H_0^2(\Omega).
 \end{equation} As in the second order case the map $ \lambda \mapsto u_\lambda(x)$ is increasing on $ (0, \lambda^*)$ and so we define the extremal solution, $ u^*$, as in the second order case.  The extremal solution is  a weak solution of $(D)_{\lambda^*}$ and for $ \lambda > \lambda^*$ there are no weak solutions. See \cite{A0,cas,AGGM} for these results. 
  The uniqueness of the extremal solution was proven for $ f(u)=e^u$ in \cite{AGGM} and for $ f(u)=(1-u)^{-2}$ \cite{cas}.    In \cite{chine} the first result was extended to the case where $f$ satisfies (R) and is log convex. We say a function $f$ is log convex provided $ t \mapsto \log( f(t))$ is convex.

  The problem $(N)_\lambda$ on general domains was studied in \cite{BG} where they obtained the same results as listed above except for the uniqueness of the extremal solution.       Some of the methods used in \cite{chine} are  inspired by \cite{BG} and so will be the techniques we use when showing the uniqueness of the extremal solution. 
  
   \subsubsection*{Systems}
  The system $(P)_{\lambda,\gamma}$, where $ f,g$ satisfy (R),   is a special case of a general system examined in \cite{Mont}.
Many of the properties one comes to expect in the second order case $(Q)_\lambda$ carry over.  The following results
are from \cite{Mont}.     Define  $ \mathcal{Q}=\{ (\lambda,\gamma): \lambda, \gamma >0 \}$ and we define
\[ \mathcal{U}:= \left\{ (\lambda,\gamma) \in \mathcal{Q}: \mbox{ there exists a smooth solution $(u,v)$ of $(Q)_{\lambda,\gamma}$} \right\}.\]

We set $ \Upsilon:= \partial \mathcal{U} \cap \mathcal{Q}$.   The curve $ \Upsilon$ is well defined and separates $ \mathcal{Q}$ into two connected components $ \mathcal{Q}$ and $ \mathcal{V}$.   We omit the various properties of $ \Upsilon$ but the interested reader should consult \cite{Mont}.   One point we mention is that if for $ x,y \in \IR^2$ we say $ x \le y$ provided $ x_i \le y_i$ for $ i=1,2$ then it is easily seen, using the method of sub and supersolutions, that if $ (0,0) < (\lambda_0,\gamma_0) \le (\lambda,\gamma) \in \mathcal{U}$ then $ (\lambda_0,\gamma_0) \in \mathcal{U}$.    Using the standard iteration procedure one easily shows
 that for each $ (\lambda,\gamma) \in \mathcal{U}$ there exists a smooth minimal solution $(u_{\lambda,\gamma},v_{\lambda,\gamma})$ of $ (Q)_{\lambda,\gamma}$ and the minimal solutions enjoy the usual monotonicity:
 if $(0,0)< (\lambda_1,\gamma_1) \le (\lambda_2,\gamma_2) \in \mathcal{U}$  then
\[ ( u_{\lambda_1,\gamma_1}, v_{\lambda_1,\gamma_1}) \le (u_{\lambda_2,\gamma_2}, v_{\lambda_2,\gamma_2}).\]

  Now for $ (\lambda^*,\gamma^*) \in \Upsilon$  there is some $ 0 < \sigma < \infty$ such that $ \gamma^* = \sigma \lambda^*$ and we can define the extremal solution $(u^*,v^*)$ at $ (\lambda^*,\gamma^*)$ by passing to the limit along the ray given by $ \gamma = \sigma \lambda$ for $ 0 < \lambda < \lambda^*$.  This limit is well defined  in the pointwise sense and it   can be shown that $ (u^*,v^*)$ is  some form of a weak solution of $ (P)_{\lambda^*,\gamma^*}$.   Our notion of a weak solution will be  more restrictive than considered in \cite{Mont}, see Remark \ref{m_remark}, and we will need to reprove this.    In the case where $ f=g$ one can use the methods from \cite{craig0} to obtain various results concerning the regularity of the extremal solution.

 \section{Main results}

  \begin{prop} \label{exist}  Suppose that $ f$ satisfies (R) or (S).    There exists some small $ \lambda_0>0$ such that for all $ 0 < \lambda < \lambda_0$ there exists a unique  smooth, stable solution $ u_\lambda$  of $(D)_\lambda$ with $ \| u_\lambda\|_{L^\infty} \le \sqrt{\lambda}$.

  \end{prop}

  \begin{proof} This is a straight forward application of the contraction mapping theorem on a suitable H\"older space.   One obtains the stability just from the fact that $ u_\lambda$ is small.

  \end{proof}

   From now on $ u_\lambda$ will refer to the minimal solution of $(N)_\lambda$ but in the context of $(D)_\lambda$ it will refer to the solution guaranteed by the above proposition.

\begin{thm}  \label{main1} (Uniqueness of solution for $(N)_\lambda, (D)_\lambda$ for small $ \lambda$) Suppose that $\Omega$ is a star shaped domain with respect to the origin in $ \IR^N$ where $ N \ge 5$.

\begin{enumerate} \item   Suppose that $f$ satisfies (R) and 
\[ \liminf_{ t \rightarrow \infty} \frac{t f(t)}{F(t)} > \frac{2 N}{N-4}.\]  Then for small $ \lambda >0$, $ u_\lambda$ is the unique smooth solution of $(D)_\lambda$ and $(N)_\lambda$.

\item  Suppose that $f$ satisfies (S).  Then for small $ \lambda >0$, $ u_\lambda$ is the unique smooth solution of  $(D)_\lambda$ and $(N)_\lambda$.

\end{enumerate}

\end{thm}

\begin{thm} \label{system_uniq} (Uniqueness of $(P)_{\lambda,\gamma})$ for small parameters)  Suppose $f(t)=g(t)=e^t$ and  $ \Omega$ is a star shaped domain with respect to the origin in $ \IR^N$ where $ N \ge 3$.    Then  $(P)_{\lambda,\gamma}$ has a unique smooth solution provided  the parameters $ 0 < \lambda, \gamma$ are sufficiently close to the origin.
\end{thm}
 The next result concerns the uniqueness of the extremal solution.   Here we need to specify what we mean by a weak solution,  which we do after stating the theorem.   Also recall that we say a function $f$ is log convex provided $ t \mapsto \log(f(t))$ is convex.   

\begin{thm} \label{ext_uniq} (Uniqueness of extremal solution) 
 
\begin{enumerate}

\item  Suppose  that $\Omega$ is a star shaped domain with respect to the origin in $ \IR^N$ where
$ N \ge 3$.   Suppose that  either $f$ and $ g$ satisfy (R) and are log convex  or that $ f$ and $ g$ satisfy (S) and are strictly convex.   Then given $ (\lambda^*,\gamma^*) \in \Upsilon $ the extremal solution $(u^*,v^*)$ is the unique weak solution of $(P)_{\lambda^*,\gamma^*}$. 

\item Suppose that $f$ is log convex and  satisfies (R) or $f$ satisfies (S) and is strictly convex.    Then the extremal solution $u^*$ is the unique weak solution of $(N)_{\lambda^*}$. 
\end{enumerate}

\end{thm}  We point out that with an extra argument, see \cite{Martel}, one can remove the strict convexity assumption on $f$. 
We now define what we mean by a weak solution.  We remark that in the case of $(N)_\lambda$ our definition coincided with the one given in \cite{BG}. 

\begin{dfn} Suppose that $ f$ and $ g$ satisfy (R). 
 \item  We say that  $u$ is a weak solution of $(N)_\lambda$ provided: $  f(u) \in L^1(\Omega)$ and that
\begin{equation} \label{weak_nav}
\int_\Omega u \Delta^2 \phi  = \int_\Omega  \lambda f(u) \phi \qquad \forall \phi \in X_{N}:= \left\{ \phi \in C^4(\overline{\Omega}): \phi=\Delta \phi =0 \;  \pOm \right\}.
\end{equation}

\item We say $ (u,v)$ is a weak solution of $(P)_{\lambda,\gamma}$ provided $f(v),g(u) \in L^1(\Omega)$ and
\begin{equation} \label{syst_weak}
\int_\Omega ( -\Delta \phi ) u = \int_\Omega \lambda f(v) \phi, \qquad  \int_\Omega (-\Delta \phi)  v = \int_\Omega \gamma \phi g(u)
\end{equation} for all $ \phi \in X_P:= \left\{ \phi \in C^2(\overline{\Omega}): \phi=0 \; \pOm \right\}$.

\end{dfn} 
In the case where $f$ and $g$ satisfy (S) we have the added condition that $ u,v \le 1$ a.e. in $ \Omega$.

\begin{remark} \label{m_remark}  At this point it is important that we mention that the notion of weak solution considered in \cite{Martel} and \cite{Mont}  requires that $ \delta f(u) \in L^1(\Omega)$, respectively $ \delta f(v), \delta g(u) \in L^1(\Omega)$, where $ \delta(x)$ is the distance from $ x$ to $\pOm$.   As mentioned previously \cite{Mont} has shown the existence of a weak solution (using his weaker notion) to $(P)_{\lambda^*,\gamma^*}$ but it is not immediately clear that this is a weak solution in our sense.  Because of this we choose to work in domains where we can prove some regularity of the extremal solution.

We remark that much of the approach we take in showing the uniqueness of the extremal solution in both the fourth order cases and the systems case is taken directly from \cite{BG} and \cite{chine}.   In \cite{BG} they developed a method capable of dealing with log convex nonlinearities in the case of the problem $(N)_\lambda$ and they used this technique to show that there are no weak solutions for $ \lambda > \lambda^*$.    This result is a major step in showing the uniqueness of the extremal solution.   In \cite{chine} the methods were extended to show the extremal solution is unique in the case of $(D)_\lambda$ on radial domains.    At essentially no extra effort this approach yields the same result for the Navier problem.

\end{remark}

\begin{remark}  In the case where $f$ is suitably subcritical one can show the existence of a second solution of $ (N)_\lambda$  (resp.  $(D)_\lambda$) for $ 0 < \lambda < \lambda^*$  (resp.  for small $ \lambda$).  Here one uses the methods from \cite{CR}. We omit the proofs. 
\end{remark}    

We mention that in a future work many of the results here are extended to equations of the form 
\[ (-\Delta)^\frac{1}{2} u = \lambda f(u) \qquad \Omega,\]
see \cite{extra}.

\section{Proofs} 

We begin with some needed results regarding the  nonlinearities.

\begin{lemma} \label{log_conv} 
\begin{enumerate}

 \item Suppose that $f$ satisfies (R) and is log convex.  Given $ \E>0$ there exists some $ 0 < \mu <1$ such that
\[ \mu^2 \left( f( \mu^{-1} t) +\E \right) \ge f(t) + \frac{\E}{2}, \]  for all $ 0 \le t$.

\item Suppose that $f$ satisfies (R) and is log convex.  Given $ 0 < \mu <1$ and $ N \ge 1$  there exists some $ k \ge 0$ such that
\[ N f(t) \le f( \mu^{-1} t ) + k, \] for all $ 0 \le t$.

\item Suppose that $f$ satisfies (R) and is log convex.  Then $ \lim_{t \rightarrow \infty} \frac{f(t) t}{F(t)}=\infty$.

 \item   Suppose that $ f$ satisfies (S).   Then $ \lim_{t \nearrow 1} \frac{f(t)}{F(t)}= \infty$. 
 
 \item Suppose that $ f$ satisfies (S).  Given $ \E>0$ there exists some $ 0 < \mu <1$ such that 
 \[ \mu \{ f( \mu^{-1} t) +\E) \} \ge f(t) + \frac{\E}{2}, \] for all $ 0 \le t \le \mu$. 
 
 \item Suppose that $f$ satisfies (R)  and is log convex.  Then $f$ is strictly convex. 
 
\end{enumerate}
\end{lemma}     In the case of a system with nonlinearities $f$ and $g$ one can take the parameters promised by the above lemma to be equal.

\begin{proof} See \cite{BG} and  \cite{chine} for the proof of 1 and 2.    \\
  3. Using L'hopital's rule 
one sees that it is sufficient to show that  $ \lim_{t \rightarrow \infty} t \frac{f'(t)}{f(t)} = \infty$.  But since $f$ is log convex we have $ t \mapsto \frac{f'(t)}{f(t)}$ is increasing and hence we are done.   \\
4.   Let $ 0 <t<1$ and we approximate $ F(t)$ using a Riemann sum with $n$ partition points and  right hand endpoints.  Doing this and using the fact that $ f$ is increasing one obtains the estimate
\[ F(t) \le \frac{(n-1)}{n} f( \frac{(n-1)}{n} t) + \frac{f(t)}{n}.\]  From this we have that $ \limsup_{t \nearrow 1} \frac{F(t)}{f(t)} \le \frac{1}{n}$ and since $n$ is arbitrary we have the desired result.  \\
5.  This follows from some simple calculus.  \\
6. Since $f$ is log convex  and increasing on $ \IR$ we can write $ f(t) = e^{\beta(t)}$ where $ \beta(t)$ is increasing and convex.  Note that by the convexity we have that $ \beta'(t)>0$ all $ t$ and so $f''(t)\ge e^{\beta(t)} \beta'(t)^2 >0$. 

\end{proof}

\textbf{Proof of Theorem \ref{main1}:}      Let $f$ satisfy (R) or (S),  $ N \ge 5$,   $ 0 < \lambda$ be small and let $ u_\lambda$ denote either the minimal solution of $(N)_\lambda$ or the solution of $(D)_\lambda$ as in the above proposition.    Suppose that $ u$ is another  solution  and set $ v:= u - u_\lambda$, so $ v $ is not identically zero.  Note that in the Navier case we have $ v \ge 0$ but in the Dirichlet case $v$ might change sign.    Then $ v$ satisfies
\begin{equation} \label{eq1}
\Delta^2 v = \lambda g(x,v)= \lambda \left\{ f(u_\lambda +v) - f(u_\lambda) \right\} \qquad \mbox{ in }   \Omega,
\end{equation} with the appropriate boundary conditions.  We now multiply (\ref{eq1}) by $ -x \cdot \nabla v$ and integrate.  In the Navier case some computations show that
\begin{equation*} 
 \int_\Omega (-x \cdot \nabla v) ( \Delta^2 v) = \frac{(N-4)}{2} \int_\Omega (\Delta v)^2 + \int_{\pOm} | \nabla (\Delta v)| | \nabla v| \nu \cdot x,
 \end{equation*}   where $ \nu$ is the outward pointing normal on $ \pOm$.  In this computation one did need to take into the account that $ -\Delta v, v \ge 0$ in $ \Omega$.     In the Dirichlet case a computation shows that
\begin{equation*}   \int_\Omega (-x \cdot \nabla v) ( \Delta^2 v) = \frac{(N-4)}{2} \int_\Omega (\Delta v)^2 + \frac{1}{2} \int_{\pOm} (\Delta v)^2 x \cdot \nu,
\end{equation*}   see \cite{taka}.   In either case the boundary integrals are  nonnegative since $ \Omega$ is star-shaped with respect to the origin and so we have 

\begin{equation*}     \frac{(N-4)}{2} \int_\Omega (\Delta v)^2 \le \int_\Omega (-x \cdot \nabla v) ( \Delta^2 v),
\end{equation*}    and using (\ref{eq1}) we have 
\begin{equation}   \label{aaaa}   \frac{(N-4)}{2} \int_\Omega (\Delta v)^2 \le \int_\Omega (-x \cdot \nabla v) \lambda \{ f(u_\lambda+v)-f(u_\lambda) \}.
\end{equation}   Define $ h(x,\tau):= f(u_\lambda(x)+\tau)-f(u_\lambda(x))$ and $ H(x,t):=\int_0^t h(x,\tau) d \tau$. Then $ H(x,t)= F(u_\lambda+t)-F(u_\lambda)-f(u_\lambda)t$, \;  $ \nabla_x H(x,t)=\{f(u_\lambda+t)-f(u_\lambda)-f'(u_\lambda)t\} \nabla u_\lambda$ and the chain rule gives $ \nabla H(x,v)= \nabla_x H(x,v) + h(x,v) \nabla v$.  So the right hand side of (\ref{aaaa}) is 
\[ \lambda \int_\Omega (-x \cdot \nabla v) h(x,v),\] which, after and integration by parts, is equal to 
\[ \lambda N \int_\Omega H(x,v) + \lambda \int_\Omega \nabla_x H(x,v) \cdot x.\]  Combining this with (\ref{aaaa}) and writing everything back in terms of $f$ and $F$ we arrive at 
\begin{eqnarray} \label{mai}
 \frac{(N-4)}{2} \int_\Omega (\Delta v)^2 & \le & \lambda N \int_\Omega \{F(u_\lambda+v)-F(u_\lambda)-f(u_\lambda)v\} \nonumber \\
 && + \lambda \int_\Omega (x \cdot \nabla u_\lambda) \{ f(u_\lambda+v) - f(u_\lambda)-f'(u_\lambda)v \}.
 \end{eqnarray}    For any $ 0 < \sigma <1$ there exists some $ C_\sigma>0$ such that  the left hand of (\ref{mai}) is bounded below by 
 \[ \frac{(N-4) \sigma}{2} \int_\Omega (\Delta v)^2 + C_\sigma \int_\Omega v^2,\] but using (\ref{eq1})  one sees that 
 \[ \int_\Omega (\Delta v)^2 = \lambda \int_\Omega \{ f(u_\lambda +v) - f(u_\lambda)\} v.\]  Putting this all together gives 
 \begin{eqnarray} \label{mayy}
 \int_\Omega \frac{(N-4) \sigma}{2} \{ f(u_\lambda+v)-f(u_\lambda)\} v + \frac{C_\sigma}{\lambda}v^2 & \le & N \int_\Omega \{ F(u_\lambda +v) - F(u_\lambda) -f(u_\lambda)v \} \nonumber \\
 && + \int_\Omega (x\cdot \nabla u_\lambda) \{ f(u_\lambda+v)-f(u_\lambda)-f'(u_\lambda) v\}.
 \end{eqnarray}  which we rewrite  as 
 \[ \int_\Omega T_\lambda(x,v) \le 0, \]  where 
 \begin{eqnarray*}
 T_\lambda(x,t)& = & \frac{(N-4) \sigma}{2} \{ f(u_\lambda+t) - f(u_\lambda) \} t + \frac{C_\sigma}{\lambda}  t^2 \\
 && - N \{ F(u_\lambda+t) - F(u_\lambda) -f(u_\lambda)t\} \\
 && - (x \cdot \nabla u_\lambda) \{ f(u_\lambda+t) - f(u_\lambda) -f'(u_\lambda)t\}.
 \end{eqnarray*}  The idea now is to obtain a contradiction by  showing that for small enough $ \lambda$ that $ T_\lambda(x,t)  >0$ for all $ x \in \Omega$ and for all $ t $  in a specific range which depends the whether $f$ satisfies (R) or (S) and whether we are in the Navier or the Dirichlet case.  Let $ S_\lambda(x,t)$ be equal to $T_\lambda(x,t)$ except that we replace the last term  $- (x \cdot \nabla u_\lambda) \{ f(u_\lambda+t) - f(u_\lambda) -f'(u_\lambda)t\}$ with $ \E_\lambda \{ f(u_\lambda+t) - f(u_\lambda) -f'(u_\lambda)t\}$ where $ \E_\lambda:= \| x \cdot \nabla u_\lambda\|_{L^\infty}$.  Note that since $ f$ is convex we have that $ T_\lambda(x,t) \ge S_\lambda(x,t)$.   We now suppose that $ f$ satisfies (R) or (S).  To show the desired positivity it is convenient to treat the cases of $ \tau$ near $ -\infty$, $ 0$ and $ \infty$ separately.   \\
 Case $ t \approx \infty$: \\   Let $ \beta$ satisfy  
 \[ \frac{2N}{N-4} < \beta < \liminf_{ t \rightarrow \infty} \frac{f(t)t}{F(t)},\]  and so there exists some $ t_0 > 1$ such that for all $ t \ge t_0-1$ we have $ f(t) t > \beta F(t)$.  Let $ 0< \lambda_0 $ be sufficiently small such  
that $ \| u_\lambda \|_{L^\infty} + \E_\lambda \le 1$ for all $ \lambda \le \lambda_0$.    So we have that $ f(u_\lambda+t) (u_\lambda+t) > \beta F(u_\lambda+t)$ for all $ t \ge t_0, x \in \lambda$ and $ \lambda \le \lambda_0$.  Pick $ \sigma $ such that $ \frac{2N}{\beta(N-4)} < \sigma <1$.  Then for $ t \ge t_0,  \lambda < \lambda_0, x \in \Omega$ we have  
\begin{eqnarray*}
S_\lambda(x,t) & \ge & f(u_\lambda+t) \left[ t \left\{ \frac{(N-4)\sigma}{2} - \frac{N}{\beta} \right\} -\E_\lambda - \frac{N u_\lambda}{\beta} \right] \\
&& + \frac{C_\sigma t^2}{\lambda} + N F(u_\lambda) - \frac{(N-4) \sigma}{2} f(u_\lambda) t.
\end{eqnarray*} 
  Now using the fact that $f$ is superlinear at $ \infty$ and since $ \frac{(N-4)\sigma}{2} - \frac{N}{\beta} >0$  one sees that 
there is some $ t_1 \ge t_0$ such that $ S_\lambda(x,t) >0$ for all 
$ x \in \Omega,  \lambda \le \lambda_0,  t \ge t_1$.    \\
Cases $ \tau \approx 0$ and $ \tau \approx -\infty$: 
 We now assume that $ - \infty <t \le t_1, \lambda \le \lambda_0, x \in \Omega$.  By the monotonicity and convexity  of $ f$ we have the lower bound 
\begin{eqnarray} \label{pzp}
 S_\lambda(x,t) & \ge & \frac{C_\sigma t^2}{\lambda} - N \{ F(u_\lambda+t) -F(u_\lambda) - f(u_\lambda)t \} \nonumber \\
&& - \{ f(u_\lambda+t) - f(u_\lambda) - f'(u_\lambda)t \}.
\end{eqnarray}  Note that all terms except the first term grow at most linearly in $t$ as $ t \rightarrow -\infty$.  
Hence there exists some $ \lambda_1 \le \lambda_0$ such that $ S_\lambda(x,t) >0$ for all $ -\infty <t \le -1$, $ \lambda < \lambda_1, x \in \Omega$.   Note this step is not needed in the Navier case.  \\
Case $ \tau  \approx 0$: \\ 
By Taylor's Theorem there exists some $ C_1 >0$ such that 
\[ | F(u_\lambda+t) -F(u_\lambda) - f(u_\lambda) t | \le C_1 t^2, \quad | f(u_\lambda+t) - f(u_\lambda) - f'(u_\lambda)t | \le C_1 t^2 ,\]  for all $ -1 \le t \le t_1,  \lambda < \lambda_0, x \in \Omega$.  Substituting this into (\ref{pzp}) and taking $ \lambda_1$ smaller if necessary we have that $ S_\lambda(x,t) >0$ for all $  0 \neq t \in [-1,t_1],  \lambda < \lambda_0, x \in \Omega$.   \\
We now assume that $ f$ satisfies (S).    Our starting point is (\ref{mayy}) and we take $ \sigma = \frac{1}{2}$.  Again we break the interval for $ t$ into 3 regions (but now the regions depends on $x$): $ t \in (1-\E - u_\lambda(x), 1-u_\lambda(x))$  (where $ \E>0$ is small),  $t \in (-1, 1-\E -u_\lambda(x))$ and $ t \in (-\infty,-1]$.  We argue as before and we use Lemma \ref{log_conv}, 4   to get the desired positivity on the first region.  For the other regions we argue as before.  We omit the details.

\hfill $ \Box$

 \textbf{Proof of Theorem \ref{system_uniq}:}    Let $ \Omega$ be a domain in $ \IR^N$ with $ N \ge 3$ and which is star shaped with respect to the origin.    Our goal it to show that the only solution of $(P)_{\lambda,\gamma}$ for $(\lambda,\gamma) \in \mathcal{Q}$ with $ \lambda^2 + \gamma^2$ small is the minimal solution.  By a symmetry argument it is sufficient to prove the result for $  0 \le \gamma \le \lambda$.  If $ \gamma = \lambda$ then $(P)_{\lambda,\gamma}$ reduces to the scalar equation, see (\ref{point}),  and we have uniqueness.  Instead of using parameters $(\lambda,\gamma)$ we prefer to use $ (\lambda,\gamma)=( \lambda, \sigma  \lambda)$  and after considering the  above comments we restrict our attention to $0 < \sigma <1$.    So with this notation we let $ (u_{\lambda, \sigma},v_{\lambda, \sigma})$ denote the minimal solution of $(P)_{\lambda, \lambda \sigma}$ where $ 0 < \lambda $ and $ 0 < \sigma <1$.      A standard argument shows that $(u_{\lambda, \sigma}, v_{\lambda,\sigma}) \rightarrow 0$ in $ C^1(\overline{\Omega}) \times C^1(\overline{\Omega})$ as $ \lambda \rightarrow 0$ uniformly in $ 0 < \sigma <1$.

 Let $ (u,v)$ denote a second solution of $ (P)_{\lambda,\sigma \lambda}$ and set
  $ u_o := u - u_{\lambda,\sigma}$, $ v_o:= v-v_{\lambda,\sigma}$.   Note these are both nonnegative and not identically zero.  We first obtain the pointwise estimates:
  \begin{equation} \label{point}
  i)  \;  v \le u, \quad  ii)  \;  \sigma u \le v, \quad  iii) \; \sigma u_o \le v_o,  \quad   iv) \;  v_o \le u_o,
  \end{equation}  where for i)- iii) there are no parameter restrictions but    in iv) the inequality will only hold for $ 0 < \lambda < \lambda_1$ and $ 0 < \sigma <1$ where $ \lambda_1>0$ is small.   \\
   Note that since $(u,v)$ is any solution that i) and ii) also hold for the minimal solution.  We now proof these.   i)  First note that we have $ -\Delta (u-v) = \lambda ( e^v - \sigma e^u)$.  Multiply this by $(u-v)_- $ and integrate over $ \Omega$ to see that

  \[ - \int_\Omega | \nabla (u-v)_- |^2 = \lambda \int_\Omega (e^v - \sigma e^u) (u-v)_- ,\]  and note the right hand side is nonnegative,  hence the left hand side is zero and we have $(u-v)_- =0 $ a.e..    \\
  ii)  Note that $ -\Delta (v - \sigma u) = \lambda \sigma (e^u - e^v)$ which is nonnegative after considering i) and after an application of   the maximum principle we see that $ v \ge \sigma u$. \\
  iii)  A computation shows that $(u_o,v_o)$ satisfy
  \begin{eqnarray}
   -\Delta u_o & = & \lambda e^{v_{\lambda,\sigma}} ( e^{v_o}-1) \qquad \Omega, \label{bare} \\
      -\Delta v_o &=& \sigma \lambda e^{u_{\lambda,\sigma}} ( e^{u_o} -1)  \qquad   \Omega,  \label{poo}
   \end{eqnarray}
    with zero Dirichlet boundary conditions.
    Set $ \Omega_0:= \left\{ x \in \Omega: v_o(x) < \sigma u_o(x) \right\}$.
    To show iii) we need to show that $ \Omega_0$ is empty, so towards a contradiction we assume its not.   Note that in $ \Omega_0$ we have
  \begin{eqnarray*}
  -\Delta (v_o- \sigma u_o) &=& \lambda \sigma \left\{ e^{u_{\lambda,\sigma}} (e^{ u_o}-1) - e^{v_{\lambda,\sigma}} ( e^{v_o} -1) \right\} \\
  & \ge & \lambda \sigma \left\{ e^{v_{\lambda,\sigma}} (e^{ u_o}-1) - e^{v_{\lambda,\sigma}} ( e^{v_o} -1) \right\} \quad \mbox{ by i)} \\
  &=& \sigma \lambda e^{v_{\lambda,\sigma}} ( e^{ {u_o}} - e^{ {v_o}} ) \\
  & \ge & \sigma \lambda e^{v_{\lambda,\sigma}} (   e^\frac{ {v_o}}{\sigma} - e^{ {v_o}} )
  \end{eqnarray*}  where the last line follows since we are in $ \Omega_0$.   Now since $ \sigma <1$ one sees the final quantity is nonnegative and hence we have that $ -\Delta ( {v_o} - \sigma {u_o}) \ge 0$ in $ \Omega_0$. Applying the maximum principle we have $ {v_o} \ge \sigma {u_o}$ in $ \Omega_0$,  which gives us the desired contradiction.     \\
  iv)   A computation shows that
  \begin{eqnarray} \label{thingy}
  -\Delta ( {u_o} - {v_o} ) &=& \lambda e^{v_{\lambda, \sigma}}( e^{ {v_o}} -1) - \sigma \lambda e^{u_{\lambda,\sigma}} ( e^{ {u_o}}-1)  \nonumber \\
  & \ge & \lambda e^{ \sigma u_{\lambda,\sigma} } ( e^{{v_o}}-1) - \sigma \lambda e^{u_{\lambda,\sigma}} ( e^{ {u_o}}-1)
  \end{eqnarray} since $ v_{\lambda,\sigma} \ge \sigma u_{\lambda,\sigma}$.

   A calculus argument shows that there exists some $ 0 <t_0$ small such that  one has
   \begin{equation} \label{cal}
    e^{\sigma t} \ge \sigma e^t \qquad \forall 0 \le t \le t_0, \; \forall 0 < \sigma <1.
    \end{equation}
  Let $ 0 < \lambda_1$ be sufficiently small such that for all $ 0 < \lambda < \lambda_1$ one has that $ \| u_{\lambda, \sigma}\|_{L^\infty} < t_0$ for all $ 0 < \sigma <1$.     We now take $ 0 < \lambda < \lambda_1$  and note that we have $ e^{\sigma u_{\lambda,\sigma}} \ge \sigma e^{u_{\lambda,\sigma}}$.  Substituting  this into (\ref{thingy}) gives

 \[ -\Delta ( {u_o} - {v_o})  \ge  \lambda \sigma e^{u_{\lambda,\sigma}}  ( e^{ {v_o}} - e^{ {u_o}}) \]  which re-arranges to
   \[ -\Delta ( {u_o} - {v_o} ) + \lambda \sigma u^{u_{\lambda,\sigma}} c(x) ( {u_o} - {v_o}) \ge 0,\]
  where $ c(x) = \frac{e^{ {u_o}} - e^{{v_o}}}{ {u_o} - {v_o}} \ge 0$ and is smooth.   The maximum principle now gives the desired result and we have completed the proofs of i) - iv).   We now return to proving uniqueness.   Let $ 0 < \lambda < \lambda_1$ and $ 0 < \sigma <1$.
   Multiply (\ref{bare}) by $-x \cdot \nabla {v_o}$ and (\ref{poo}) by $ -x \cdot \nabla {u_o}$  and integrate to obtain
   \begin{eqnarray} \label{first}
   \int_\Omega \Delta u_o ( x \cdot \nabla v_o) & = &  \lambda N \int_\Omega e^{v_{\lambda,\sigma}} ( e^{v_o} - v_o-1) \nonumber \\
   && + \lambda \int_\Omega e^{v_{\lambda,\sigma}} ( x \cdot \nabla v_{\lambda,\sigma}) ( e^{v_o} -v_o -1)
   \end{eqnarray}

   \begin{eqnarray} \label{second}
   \int_\Omega \Delta v_o ( x \cdot \nabla u_o) & = &  \lambda N \sigma \int_\Omega e^{u_{\lambda,\sigma}} ( e^{u_o} - u_o-1) \nonumber \\
   && + \lambda \sigma \int_\Omega e^{u_{\lambda,\sigma}} ( x \cdot \nabla u_{\lambda,\sigma}) ( e^{u_o} -u_o -1).
   \end{eqnarray}  A computations shows that $ \Delta ( x \cdot \nabla v_o) = 2 \Delta v_o + x \cdot \nabla ( \Delta v_o)$.   Using this and a integration by parts shows that

   \begin{equation} \label{three}
   \int_\Omega \Delta u_o ( x \cdot \nabla v_o) + \Delta v_o ( x \cdot \nabla u_o) = (N-2) \int_\Omega \nabla u_o \cdot \nabla v_o +   \int_{\pOm} | \nabla u_o | | \nabla v_o| x \cdot \nu.
   \end{equation}   Note for the boundary term we have used the fact that $ u_o, v_o \ge 0$ in $ \Omega$.  Adding (\ref{first}) and (\ref{second}) and using (\ref{three}) gives

   \begin{eqnarray} \label{four}
   (N-2) \int_\Omega \nabla u_o \cdot \nabla v_o & \le &   \lambda N \int_\Omega e^{v_{\lambda,\sigma}} ( e^{v_o} - v_o-1) \nonumber \\
   && + \lambda \int_\Omega e^{v_{\lambda,\sigma}} ( x \cdot \nabla v_{\lambda,\sigma}) ( e^{v_o} -v_o -1) \nonumber \\
    && + \lambda N \sigma \int_\Omega e^{u_{\lambda,\sigma}} ( e^{u_o} - u_o-1) \nonumber \\
   && + \lambda \sigma \int_\Omega e^{u_{\lambda,\sigma}} ( x \cdot \nabla u_{\lambda,\sigma}) ( e^{u_o} -u_o -1).
   \end{eqnarray}

      Now we know that $ -\Delta u_o, -\Delta v_o \ge 0$  and we also have $ \sigma u_o \le v_o \le u_o$.  From this we see that

   \begin{equation} \label{five}
   \int_\Omega \nabla u_o \cdot \nabla v_o = \int_\Omega (-\Delta u_o) v_o \ge \sigma \int_\Omega | \nabla u_o|^2 \ge \sigma \lambda_1(\Omega) \int_\Omega u_o^2,
   \end{equation}   and similarly one shows

   \begin{equation} \label{six}
   \int_\Omega \nabla u_o \cdot \nabla v_o  \ge \lambda_1(\Omega) \int_\Omega v_o^2,
   \end{equation}   where $ \lambda_1(\Omega)$ denotes the first eigenvalue of $ -\Delta $  in $ H_0^1(\Omega)$.   Using (\ref{bare}) and (\ref{poo}) one also sees that

      \begin{equation} \label{seven}
   \lambda \int_\Omega e^{v_\lambda} (e^{v_o} -1) v_o = \int_\Omega \nabla u_o \cdot \nabla v_o = \lambda \sigma \int_\Omega e^{u_\lambda} (e^{u_o}-1) u_o.
   \end{equation}

   The idea is to now break the left hand side of (\ref{four}) into four equal parts and use
   (\ref{five}), (\ref{six}) and (\ref{seven}) to rewrite (\ref{four}). We now take   $ 0< \lambda <\lambda_1$ sufficiently small such that   $ e^{u_{\lambda,\sigma}}, e^{v_{\lambda,\sigma}}<2$ for all $ 0 < \sigma <1$.    Doing this we obtain an inequality of the form
   \begin{equation} \label{nine}
    \int_\Omega \sigma  \left\{ \frac{ u_o^2}{\lambda} + (e^{u_o}-1)u_o - C(u^{u_o}-u_o-1) \right\}  +  \left\{\frac{v_o^2}{ \lambda} + ( e^{v_o} -1)v_o - C( e^{v_o}-v_0-1) \right\} dx  \le 0
   \end{equation} where $ C=C(N)>0$.  One easily sees that for $ 0 < \lambda $ sufficiently small that the integrand in (\ref{nine}) is positive on $ \{ ( u_o,v_o): u_o, v_o \ge 0\} \backslash \{(0,0)\}$.     Hence we have $ u_o=v_o=0$ and so $(u,v)=(u_{\lambda,\sigma}, v_{\lambda,\sigma})$. \\

 \hfill $ \Box$

\textbf{Proof of Theorem \ref{ext_uniq}:}     1. We first show that the extremal solution is a weak solution.  Let $ (u^*,v^*)$ denote the extremal solution corresponding to the parameters $ (\lambda^*, \gamma^*)$.    Using the techniques from \cite{Mont} one sees that $ (u^*,v^*)$ is a weak solution of $(P)_{\lambda^*,\gamma^*}$ except for possibly the integrability conditions.   To obtain these we obtain estimates on the minimal solutions along the ray through origin and through $ (\lambda^*,\gamma^*)$.  
Let $ (\lambda,\gamma)$ lie on this ray  and let  $(u,v)$ denote the minimal solution of $(P)_{\lambda,\gamma}$.  Multiply $ -\Delta u = \lambda f(v)$ by $ -x \cdot \nabla v$ and $ -\Delta v = \gamma g(u)$ by $ -x \cdot \nabla u$ and add the inequalities and integrate over $ \Omega$ to arrive at
\[ \int_\Omega x \cdot \nabla v \Delta u +  x \cdot \nabla u \Delta v = \int_\Omega \lambda f(v) (-x \cdot \nabla v) + \gamma g(u) (-x \cdot \nabla u),\]
and arguing as in (\ref{first}), (\ref{second}) and (\ref{three}) one sees that we have
\[ (N-2) \int_\Omega \nabla u \cdot \nabla v \le \lambda N \int_\Omega F(v) + \gamma N \int_\Omega G(u).\]  Now using the equation for $(u,v)$ we see that
\[ \int_\Omega \frac{\lambda (N-2)}{2} f(v) v + \frac{\gamma(N-2)}{2} g(u) u \le \int_\Omega \lambda N F(v) + \gamma N G(u).\]  From Lemma   \ref{log_conv}  we see that $ f(t) t $ dominates $F(t)$ for $ t$ near $ \infty$ (resp. near $1$)  in the case where $ f$ satisfies (R) and is log convex (resp. $f$ satisfies (S)).   One has the same for $ g$ and $G$.   From this we conclude that we have uniform bounds on $ \int_\Omega f(v) v$ and $ \int_\Omega g(u)u$ along the given ray and so passing to limits we have   $ v^* f(v^*),  u^* g(u^*) \in L^1(\Omega)$ and we also have the desired $H^1_0(\Omega)$ bound.     

We now show that the extremal solution is the unique solution.   Assume that $ 0 < \sigma < \infty$ is such that $ \gamma^*=\sigma \lambda^*$ and that $ (u,v)$ is a second weak solution of  $(P)_{\lambda^*,\gamma^*}$.  For simplicity we assume that $ \lambda^*=1$.  By the minimality of the extremal solution we see that $ (u,v) \ge (u^*,v^*)$ a.e. in $ \Omega$ and we have that $ u \neq u^*$ and $ v \neq v^*$.  We now assume that $ f$ and $ g$ satisfy (R) and are log convex.  
  Define
\[ z_1:= \frac{u^*+u}{2}, \quad z_2:= \frac{v^*+v}{2},  \] and note that $ z_1$ and $ z_2$ are weak solutions of
\[ -\Delta z_1 = f(z_2) + h_1(x), \qquad -\Delta z_2 = \sigma g(z_1) + \sigma h_2(x),\qquad \Omega,\] with $ z_1=z_2=0 $ on $ \pOm$ where
we define $h_i$ in a  moment.  \\
 Define $ \Omega_1=\{ x \in \Omega: v(x),v^*(x),u(x),u^*(x) \in \IR\}$  and note that $ \Omega \backslash \Omega_1$ is a set of measure zero.     We define
\[ h_1(x)= \frac{f(v^*) + f(v)}{2} - f( \frac{v^*+v}{2}) \quad x \in \Omega_1,\]
\[ h_2(x)= \frac{g(u^*)+ g(u)}{2} - g( \frac{u^*+u}{2})  \quad x \in \Omega_1,\]  and we set both to be zero otherwise.
 Note that since $f$ and $g$ are convex we have that $ 0 \le h_i $ a.e. in $ \Omega$ and since $(u,v)$ and $(u^*,v^*)$ are weak solutions we have $ h_i \in L^1(\Omega)$.     Since $f$ and $g$ are strictly convex (either by hypothesis or by Lemma \ref{log_conv}, 6)  and so we have $h_i$ different from zero on a set of positive measure.    Let $ \chi_i$ be weak solutions of $ -\Delta \chi_1 = h_1$ and $- \Delta \chi_2 = \sigma h_2$ in $ \Omega$ with zero boundary conditions and let $ -\Delta \phi = 1$ in $ \Omega$ with $ \phi=0$ on $ \pOm$.  By Hopf's Lemma there is some small $ \E>0$ such that $ \chi_1 \ge \E \phi$ and $ \chi_2 \ge \E \sigma \phi$ in $ \Omega$.   We now set
 \[ \tau_1:= z_1 + \E \phi -\chi_1, \qquad \tau_2:= z_2 + \sigma \E \phi - \chi_2,\] and note that $ \tau_i \le z_i$ in $ \Omega$.  A computation shows that  $ \tau_1$ and $ \tau_2$ are weak solutions of $ -\Delta \tau_1 = f(z_2) + \E$ in $ \Omega$ and $ -\Delta \tau_2 = \sigma( g(z_1)+\E)$ in $ \Omega$ with $ \tau_i=0$ on $ \pOm$.   Since $ z_i \ge \tau_i$ one sees that $ \tau_i$ are weak supersolutions, in a suitable sense (see the proof of the claim), of $ -\Delta \tau_1 \ge f(\tau_2) +\E$ in $ \Omega$ and $ -\Delta \tau_2 \ge \sigma( g(\tau_1)+\E)$ in $ \Omega$ with $ \tau_i=0$ on $ \pOm$.  We now use the following claim which we prove in a moment.  \\
   \textbf{Claim:}  There exists $ 0 \le w_i$ smooth such that
 \[ -\Delta w_1 = f(w_2) + \frac{\E}{2}, \qquad -\Delta w_2 = \sigma( g(w_1) + \frac{\E}{2}) \qquad \Omega, \] with $ w_i=0$ on $ \pOm$.   Let $ w_i$ be as in the claim and   pick $ \alpha >0$ but sufficiently small such that $ \alpha w_1 \le \frac{\E \phi}{2}$ and $ \alpha w_2 \le \frac{\sigma \E \phi}{2}$ in $\Omega$, which is not an issue since $ w_i$ is smooth.  Set $ \overline{w_1}= w_1 + \alpha w_1 - \frac{\E \phi}{2}$ and $\overline{w_2}= w_2+ \alpha w_2 - \frac{\sigma \E \phi}{2}$.   Note that $ \overline{w_i} \le w_i$ in $ \Omega$ and also note that a computation shows that 
\[ -\Delta \overline{w_1} \ge (1+\alpha) f( \overline{w_2}), \qquad -\Delta \overline{w_2} \ge (1+\alpha) \sigma g( \overline{w_1}) \qquad \Omega,\] where $ \overline{w_i}=0 $ on $ \pOm$.    The maximum principle shows that $ \overline{w_i} \ge 0$.  Now one uses a standard iteration argument to obtain a bounded solution, which is smooth after applying standard elliptic regularity theory, to $(P)_{1+\alpha, \sigma(1+\alpha)}$ which contradicts the fact that we assumed $ \lambda^*=1$.  To finish the proof we need only prove the claim and for this we switch notation slightly so as to cut down on the indices.  Suppose that $ \E>0$  and we have $ 0 \le u_0,v_0 \in L^1(\Omega)$ are weak solutions of $ -\Delta u_0 = k_0(x)  $ and $ -\Delta v_0 = k_1(x)$ in $ \Omega$ with $ u_0=v_0 =0 $ on $ \pOm$ where $ 0 \le k_i \in L^1(\Omega)$ and $ k_0(x) \ge f(v_0)+\E$ and $ k_1(x) \ge \sigma\{ g(u_0)+\E\} $ in $ \Omega$.  (I am using this somewhat restrictive notion of a weak supersolution since this is sufficient for our needs).  To prove the claim we need to now show the existence of bounded solutions of $ -\Delta \tilde{u} = f(\tilde{v})+\frac{\E}{2}$ and $ -\Delta \tilde{v}= \sigma( g(\tilde{u})+ \frac{\E}{2})$ in $ \Omega$ with $ \tilde{u}=\tilde{v}=0$ on $ \pOm$.   Let $ 0 < \mu <1$ be as promised from Lemma \ref{log_conv}, 1  and then let $ k$ be from 1 (ii) of the same lemma. We let $ u_i$ and $v_i$ for $ i=1,2,3$ denote  weak solutions   of
\[ -\Delta u_1 = \mu( f(v_0)+\E), \qquad -\Delta v_1 = \mu \sigma( g(u_0)+\E) \qquad \Omega, \]
\[ -\Delta u_2 = \mu( f(v_1)+\E), \qquad  -\Delta v_2 = \mu \sigma( g(u_1)+\E) \qquad \Omega, \]
\[  -\Delta u_3 = \mu( f(v_2)+\E), \qquad   -\Delta v_3 = \mu \sigma( g(u_2)+\E) \qquad \Omega, \] all with zero Dirichlet boundary conditions.  By the 
weak maximum principle we have that $ 0 \le u_3 \le u_2 \le u_1 \le \mu u_0$  and $ 0 \le v_3 \le v_2 \le v_1 \le \mu v_0$ in $ \Omega$.
 Let $ -\Delta \phi=1$ in $ \Omega$ with $ \phi=0$ on $ \pOm$.    Let $ T>0$ which we pick later.  Note that
\begin{eqnarray*}
-\Delta (u_1+T \phi) &=& T + \mu( f(v_0)+\E) \\
& \ge & T + \mu(  f( \frac{ v_1}{\mu})+\E) \\
& \ge & T +    \mu ( N f(v_1) -k+\E ) \\
&=& T+ \mu \E - \mu k - N \E \mu + N( \mu f(v_1) + \mu \E) \\
&=& T+ \mu \E - \mu k - N \E \mu + N (-\Delta u_2)
\end{eqnarray*}  and so if we take $ T $ big enough such that $ T+ \mu \E - \mu k - N \E \mu  \ge 0$ then we have that $N u_2 \le u_1 + T \phi$ in $\Omega$.  A similar calculation shows that
\[ -\Delta (v_1 + T \phi) \ge T + \mu \sigma \E - \mu \sigma k - N \mu \sigma \E + N ( -\Delta v_2) \qquad \Omega, \]  and so by taking $ T$ larger if necessary we also have that $ N v_2 \le v_1 + T \phi $ in $ \Omega$.    Now since $ f$ and $g$ are log convex we can write $ f(t) = e^{\gamma_1(t)}$ and $ g(t)= e^{\gamma_2(t)}$ where $ \gamma_i$ is convex and increasing with $ \gamma_i(0)=0$.   So we have that
\[ f(v_2) \le e^{\gamma_1( \frac{v_1+T \phi}{N})},\]  and note that
\begin{eqnarray*}
\gamma_1( \frac{v_1+T\phi}{N}) &=& \gamma_1(  \frac{1}{N} v_1 + (1 - \frac{1}{N}) \frac{T \phi}{(N-1)}) \\
& \le & \frac{\gamma_1(v_1)}{N} + ( 1 - \frac{1}{N}) \gamma_1( \frac{T \phi}{(N-1)})
\end{eqnarray*}  and from this we obtain that
\[ f(v_2)^N \le e^{\gamma_1(v_1)} e^{(N-1) \gamma_1( \frac{T \phi}{N-1})} \le f(v_0) e^{(N-1) \gamma_1( \frac{T \phi}{N-1})},\]  which shows that $ f(v_2) \in L^N(\Omega)$.  A similar calculation shows that $ g(u_2) \in L^N(\Omega)$ and hence by elliptic regularity theory we have that $ u_3,v_3 $ are bounded and note that since $ u_3 \le u_2$ and $ v_3 \le v_2$ we see that they satisfy $ -\Delta u_3 \ge \mu (f(v_3)+\E) $ and $ -\Delta v_3 \ge \sigma \mu ( g(u_3) +\E) $ in $ \Omega$ and we can apply a standard iteration argument to obtain smooth solutions 
$ \underline{u}$ and $ \underline{v}$ of $ -\Delta \underline{u} = \mu( f(\underline{v})+\E)$ and $ -\Delta \underline{v} = \sigma \mu ( g(\underline{u})+\E)$ in $ \Omega$ with $ \underline{u}=\underline{v}=0$ on $ \pOm$.  We now set $ \delta_1 = \mu \underline{u}$ and $ \delta_2 = \mu \underline{v}$. Then a computation shows that
\[ -\Delta \delta_1 = \mu^2( f(\underline{v})+\E) = \mu^2( f( \frac{\delta_2}{\mu}) +\E) \ge f(\delta_2)+\frac{\E}{2}\qquad \Omega,\] where the last inequality follows from Lemma \ref{log_conv}, 1.  A similar calculation shows that
\[ -\Delta \delta_2 \ge \sigma( g(\delta_1)+ \frac{\E}{2}) \qquad \Omega, \]  and we now obtained the desired result after a standard iteration argument.     \\
We now assume that $ f$ and $ g$ satisfy (S). Everything carries through as in the previous case except for the proof of the Claim.   Suppose that we have weak supersolutions $(u,v)$ of  
\[ -\Delta u \ge f(v) +\E, \qquad -\Delta v \ge \sigma( g(u) +\E) \qquad \Omega,\]  with  Dirichlet boundary conditions.     Let $ 0 < \mu <1$ be as promised from Lemma \ref{log_conv}, 5.  Set $ \overline{u}= \mu u$ and $ \overline{v}= \mu v$.   The first thing to notice is that $ u,v \le \mu $ a.e. in $ \Omega$.  A computation and Lemma \ref{log_conv}, 5,  show that $(\overline{u},\overline{v})$ is a weak supersolution (but bounded away from $1$) of 
\[ -\Delta \overline{u} \ge f(\overline{v}) + \frac{\E}{2}, \qquad -\Delta \overline{v} \ge \sigma \{ g(\overline{u}) + \frac{\E}{2} \}, \qquad \Omega\]  and so we can now apply a monotone iteration to obtain the desired result.

2.  It is known that $ u^*$ is a weak solution of $(N)_{\lambda^*}$ see \cite{BG}.  In fact the extremal solution enjoys the added regularity, $ f(u^*) \in L^2(\Omega)$, see \cite{craig1}.   We now show that $ u^*$ is the unique weak solution of $(N)_{\lambda^*}$ and the proof is very similar to the proof of 1 so we will be somewhat brief.  If no boundary conditions are given then it is understood they are Navier.   We suppose that $ u$ is a second weak solution of $(N)_{\lambda^*}$ and for simplicity we assume that $ \lambda^*=1$.  By the minimality of $ u^*$ we have that $  u^* \le u $ a.e. in $ \Omega$ and they differ on a set of positive measure.  Set $ z:= \frac{u^*+u}{2}$ and note that $ z$ is a weak solution of $ \Delta^2 z = f(z) + h(x) $ in $ \Omega$ where $ h(x)=2^{-1} \{ f(u^*) + f(u) \} - f(z)$ on $ \Omega_0:=\{ x \in \Omega:  u^*(x),u(x) \in \IR\}$ and where $ h(x)=0$ otherwise.  Again we have $ | \Omega \backslash \Omega_0 |=0$ and note that $ 0 \le h$ on $ \Omega$.    Since $f$ is strictly convex we have that $ h$ is positive on set of positive measure.   Let $ \chi$ be a weak solution of $ \Delta^2 \chi =h$ in $ \Omega$ and $ \Delta^2 \phi =1$ in $ \Omega$.  By Hopf's lemma (smooth out $h$ if necessary) there is some $ \E>0$ such that $ -\Delta( \chi -\E \phi) \ge 0$ in $ \Omega$ and so the maximum principle shows that $ \chi \ge \E \phi$ in $ \Omega$. 
 Set $ \tau := z +\E \phi - \chi$ and note that $ \tau \le z$ a.e. in $ \Omega$.  Also note that $ \tau$ is a weak solution of  $ \Delta^2 \tau = f(z) +\E$ in $ \Omega$ and so $ \tau \ge 0$ a.e. in  $\Omega$.  Since $ z \ge \tau$ we see that $ \tau $ is a weak supersolution of $ \Delta^2 \tau \ge f( \tau) +\E$ in $ \Omega$ (here we are using the analogous notion of a weak supersolution as in 2).    \\
 Claim: there exists some smooth function $ 0 < w$ such that $ \Delta^2 w = f(w)+ \frac{\E}{2}$ in $ \Omega$.  \\
 Take $ \alpha >0$ but small enough such that $ \alpha w \le \frac{\E \phi}{2}$ in $ \Omega$.   Set $ \overline{w}=w+ \alpha w - \frac{\E \phi}{2}$ and note that $ \overline{w} \le w$.  Then we see that $ \overline{w}$ satisfies $ \Delta^2 \overline{w} \ge (1+\alpha) f( \overline{w})$ in $ \Omega$ and so $ \overline{w} \ge 0$ in $ \Omega$.   Using the usual iteration argument shows the existence of a smooth solution to $ \Delta^2 \tilde{w}= (1+\alpha) f( \tilde{w})$ in $ \Omega$ which contradicts the fact  that $ \lambda^*=1$.    The only thing left to show is the claim.  The proof is very similar to the proof of the analogous claim in 1, so we omit the details. 

\hfill $ \Box$

\end{document}